\documentclass{amsart}

\newtheorem{theorem}[subsubsection]{Theorem}
\newtheorem{lemma}[subsubsection]{Lemma}

\newtheorem{corollary}[subsubsection]{Corollary}

\theoremstyle{definition}

\theoremstyle{remark}
\newtheorem{remark}[subsubsection]{Remark}

\DeclareMathAlphabet{\mathpzc}{OT1}{pzc}{m}{it}
\usepackage{amsmath}
\usepackage{amssymb}
\usepackage{euscript}
\usepackage{array}
\usepackage[all]{xy}
\usepackage{amscd}

 \newcommand{\RR}{\mathbb{R}}
 \newcommand{\FF}{{\mathbb F}}
 
 \newcommand{\ZZ}{{\mathbb Z}}
 \newcommand{\QQ}{{\mathbb Q}}
 \newcommand{\CC}{{\mathbb C}}
 \newcommand{\NN}{{\mathbb N}}

 \newcommand{\GG}{{\mathbb G}}

\newcommand{\TT}{{\mathbb T}}
\newcommand{\LL}{{\mathbb L}}

\renewcommand{\a}{\alpha}   
  \renewcommand{\th}{\theta}
  
  \renewcommand{\th}{\theta}
    
 \newcommand{\s}{\sigma} 
  
\newcommand{\G}{\Gamma}

\begin{document}
\title[Independence of gamma and zeta values]{Algebraic independence
of arithmetic gamma values and Carlitz zeta values }

\author[Chang]{Chieh-Yu Chang}
\address{National Center for Theoretical Sciences, Mathematics Division,
National Tsing Hua University, Hsinchu City 30042, Taiwan
  R.O.C.}
\address{Department of Mathematics, National Central
  University, Chung-Li 32054, Taiwan R.O.C.}
\email{cychang@math.cts.nthu.edu.tw}

\author[Papanikolas]{Matthew A. Papanikolas} \address{Department of Mathematics,
  Texas A{\&}M University, College Station, TX 77843-3368, USA}
\email{map@math.tamu.edu}

\author[Thakur]{Dinesh S. Thakur} \address{Department of Mathematics,
  University of Arizona, Tucson, AZ 85721, USA}
\email{thakur@math.arizona.edu}

\author[Yu]{Jing Yu} \address{Department of Mathematics, National Taiwan University,
Taipei City 106, Taiwan R.O.C.}  \email{yu@math.ntu.edu.tw }

\date{August 31, 2009}

\thanks{The first author was supported by an NCTS postdoctoral fellowship.
The second author was supported by NSF Grant DMS-0600826. The third
author was supported by NSA grant H98230-08-1-0049. The fourth
author was supported by NSC Grant No.\ 96-2628-M-007-006.}


\begin{abstract}
  We consider the values at proper fractions of the arithmetic gamma
  function and the values at positive integers of the zeta function
  for $\FF_q[\theta]$ and provide complete algebraic independence results for them.
\end{abstract}

\maketitle
\section{Introduction}
In this paper, we determine all of the algebraic relations among
special values of two important functions in function field
arithmetic, namely the arithmetic gamma function and the zeta
function associated to $\FF_q[\theta]$.  For more background on
function field arithmetic and on the properties of these functions,
we refer to [G96,  T04] and references there. We give only a
brief introduction that is relevant here.

Let us write $A=\FF_q[\theta]$, $k=\FF_q(\theta)$,
$k_{\infty}=\FF_q((1/\theta))$, the completion of $k$ at its usual
infinite place, and $\CC_{\infty}$ for the completion of an
algebraic closure of $k_{\infty}$. It is well-known that these are
good analogues of $\ZZ$, $\QQ$, $\RR$, and $\CC$ respectively.  Let
$A_+$ denote the set of monic polynomials in $A$. This is considered
to be an analogue of the set $\NN$ of positive integers.

By (Carlitz) zeta values, we mean
$$\zeta_{C}(s)=\sum_{a\in A_+}\frac{1}{a^s}\in k_{\infty}, \ \ s\in \NN.$$
These values, first considered by Carlitz, are the special values at
positive integers of the zeta function studied by Goss.

Let $$D_{n}:=\prod_{i=0}^{n-1}( \theta^{q^{n}}-\theta^{q^{i}}  ), \
\ \ \overline{D_{n}}:= D_{n}/(\theta^{\deg D_{n}} ).$$ The Carlitz
factorial of $n$ is defined to be $\prod D_i^{n_i}\in \FF_q[\th]$
for $n=\sum n_iq^i\in \NN$, $0\leq n_i<q$, and the interpolation of
its unit part for $n \in \ZZ_p$, due to Goss \cite{G80}, is
$$ n!:=\prod \overline{D_{i}} ^{n_{i}}\in k_{\infty} \hbox{ for }n=\sum n_{i} q^{i}  ,\hbox{ }0\leq n_{i}<q. $$
By (arithmetic or Carlitz-Goss) gamma values, we mean values at
proper fractions of this function.

Our goal in this paper is to determine completely the algebraic
relations among gamma and zeta values.  We mention that there are
many parallels (see \cite{G96} and \cite{T04}) between known facts
about classical zeta and gamma values and their function field
counterparts, such as Euler's theorem, prime factorization,
interpolations at all finite primes, functional equations,
Gross-Koblitz formulas, and Chowla-Selberg formulas.  In fact, for
$\FF_q[\theta]$-arithmetic, there are two types of gamma functions,
the arithmetic one above dealing with the cyclotomic theory of usual
roots of unity corresponding to constant field extensions and the
geometric one dealing with the cyclotomic theory of Carlitz-Drinfeld
cyclotomic extensions.  For a unified treatment of these two
together with the classical gamma function, see \cite[Sec.
4.12]{T04}.

Very briefly, the development of the special value theory is the following.

For the classical gamma function the only gamma values (at proper
fractions) known to be transcendental have denominators $2$, $4$ or
$6$, with $\Gamma(1/2)=\sqrt{\pi}$, and with the ones with
denominators $4$ and $6$ related to periods of elliptic curves with
complex multiplications by Gaussian or Eisenstein integers, via the
Chowla-Selberg formula. There is an algebraic independence result
concerning these values due to Chudnovsky. (The Beta value theory is
much better developed by results of Wolfart-W\"ustholz).

The third author (in his 1987 thesis, see \cite{T91}) proved
analogues of the first formula (with $2\pi i$ replaced by
$\tilde{\pi}$, the period of Carlitz module) and of the
Chowla-Selberg formula for the arithmetic gamma function, resulting
in a parallel transcendence and independence statement by results of
the fourth author and Thiery on transcendence of periods.  In
\cite{T96}, using automata methods, the transcendence of gamma
values with any denominator, but with some restrictions on
numerators, was established.  These restrictions were then removed
by Allouche \cite{Al96}. Mend\`{e}s France and Yao \cite{MFY97} gave
an easier automata proof and finally the third author completely
determined all transcendental monomials in gamma values.  (See
\cite[Sec.~11.3]{T04}; we also take this opportunity to correct
`turns out to be a trivial monomial' in the last but one paragraph
of p.~349, by adding `after translating by integers appropriately to
apply Theorem 4.6.4' as explained on p.~351.)  But the more general
algebraic dependence question remained.

As for the geometric gamma function, in \cite{T91}, some very
special values were related to the periods of Drinfeld modules, thus
establishing their transcendence by the results of Wade and the
fourth author. Anderson, Sinha, Brownawell and the second author
\cite{S97}, \cite{BP02}, \cite{ABP04}, \cite{P08} connected all
values at proper fractions of geometric gamma to periods and
quasi-periods of certain $t$-motives of Anderson \cite{A86}
(generalizations of Drinfeld modules) and at the same time provided
strong transcendence tools to completely determine \cite{ABP04} all
algebraic relations among them.  This paper does the same for the
arithmetic gamma function, by using these new tools and connecting
these values to the periods and quasi-periods of appropriate
$t$-motives (see Section \ref{subsec  gamma}), thus bypassing the
automata theory.

The main result of the present paper (Theorem~\ref{main result})
considers arithmetic gamma values and zeta values simultaneously and
furthermore determines all algebraic relations among them.  The
earlier results of the fourth author \cite{Y91} on transcendence of
zeta values, already surpassing the parallel classical results, have
recently been further improved \cite{CY07}, \cite{CPY08a} to
complete algebraic independence results for zeta values.  Here we
show that these techniques generalize to give algebraic independence
of both arithmetic gamma and zeta values together.

We briefly mention some additional avenues of research one can now
pursue in light of Theorem~\ref{main result}.  In \cite{CPY08a},
specific techniques inspired from \cite{Ch09} are introduced to deal
with varying $q$, i.e.\ to obtain algebraic independence results for
zeta values at positive integers with varying constant fields. This
method certainly can also work for gamma values, in particular to
determine all algebraic relations among special arithmetic gamma
values as the constant fields vary in the same characteristic. This
more complicated question will not be treated here however.  We note
that in another paper \cite{CPY09b}, the first, second, and fourth
author have also established algebraic independence of geometric
gamma and zeta values taken together.  We leave the question of
algebraic independence of arithmetic gamma and geometric gamma
values taken together to a later work.  Note that for the special
case  $q=2$, the geometric gamma values in question are algebraic
multiples of $\tilde{\pi}$ and the zeta value $\zeta_{C}(n)$ is a
rational multiple of $\tilde{\pi}^{n}$ for each $n\in \NN$, and so
the present paper covers the algebraic relations of all three
together completely in this case.

Overall the present results bring the special value theory of the
arithmetic gamma function and zeta function for $A=\FF_q[\theta]$ to
a very satisfactory state; however, similar questions for (i)
$v$-adic interpolations (see \cite[Sec.~11.3]{T04} for very partial
results about $v$-adic gamma values using automata methods and
\cite{Y91} for transcendence of $v$-adic zeta values), (ii)
generalizations to other rings `$A$' in the setting of Drinfeld
modules \cite[4.5, 8.3]{T04}, (iii) values of the two variable gamma
function of Goss \cite[4.12]{T04}, \cite[Sec.~8]{T91}, \cite{G88},
are still open.

A main tool here for proving algebraic independence is a theorem of
the second author \cite{P08}, which is a function field version of
Grothendieck's conjecture on periods of abelian varieties. The
$t$-motives related to special arithmetic gamma values  have
``complex multiplication'' by constant field extensions. This fact
enables us to show that the associated Galois group by Tannakian
duality is the  Weil restriction of scalars of $\GG_{m}$ from the
constant field extension in question, hence a torus. On the other
hand, according to Chang-Yu \cite{CY07}, Galois groups of the
$t$-motives related to Carlitz zeta values are always extensions of
$\GG_{m}$ by vector groups.  The Galois group of the direct sum of
these two types of $t$-motives can be shown to be an extension of a
torus by a vector group.  Adding the dimension of the torus with
that of the vector group in question proves the desired algebraic
independence result.  Thus, the story of arithmetic gamma values
unfolds through $t$-motives having arithmetic (cyclotomic) CM in
this paper, just as $t$-motives having geometric (cyclotomic) CM
provide the proper setting for special geometric gamma values
\cite{ABP04}, \cite{CPY09b}.

For the rest of the paper, we will only consider arithmetic gamma
values and Carlitz zeta values as defined above.

\textbf{Acknowledgments:} The authors thank the Arizona Winter
School 2008 on `Special Functions and Transcendence,' where this
collaboration began. The first author thanks NCTS for supporting him
to visit Texas A{\&}M University. He also thanks Texas A{\&}M
University for their hospitality. The third author thanks the von
Neumann Fund and the Ellentuck Fund for their support during his
stay at the Institute for Advanced Study at Princeton, where he
worked on this paper.

\section{$t$-motives and Periods }
\subsection{Notations.}
\subsubsection{Table of symbols.}${}$
\\$\mathbb{F}_{q}:=$ the finite field of $q$ elements, $q$ a power
of a prime number $p$.
\\$\th,t:=$ independent variables.
\\$A:=\mathbb{F}_{q}[\theta]$, the polynomial ring in the variable
$\theta$ over $\mathbb{F}_{q}$.
\\$A_{+}:=$ the set of monic polynomials of $A$.
\\$k:=\mathbb{F}_{q}(\theta)$, the fraction field of $A$.
\\$k_{\infty}:=\mathbb{F}_{q}((\frac{1}{\theta}))$, the completion of
$k$ with respect to the place at infinity.
\\$\overline{k_{\infty}}:=$ a fixed algebraic closure of
$k_{\infty}$.
\\$\overline{k}:=$ the algebraic closure of $k$ in
$\overline{k_{\infty}}$.
\\$\mathbb{C}_{\infty}:=$ the completion of $\overline{k_{\infty}}$
with respect to canonical extension of the place at infinity.
\\$|\cdot|_{\infty}:=$ a fixed absolute value for the completed field
$\mathbb{C}_{\infty}$.
\\$\mathbb{T}:=\{f\in \mathbb{C}_{\infty}[[t]]: f\ \textnormal{converges for\ }
|t|_{\infty}\leq_{}1 \}$, the Tate algebra.
\\$\mathbb{L}:=$ the fraction field of $\mathbb{T}$.
\\${\GG}_{a}:=$ the additive group.
\\$\hbox{GL}_r/F:=$ for a field $F$, the $F$-group scheme of invertible rank $r$  matrices.
\\${\GG}_{m}:=\hbox{GL}_{1}$, the multiplicative group.

\subsubsection{Twisting.}  For $n\in {\ZZ}$, given
$f=\sum_{i} a_{i}t^{i}\in \mathbb{C}_{\infty}((t))$,  define the
twist of $f$ by
${\s}^{n}(f):=f^{(-n)}=\sum_{i}a_{i}^{q^{-n}}t^{i}$. The twisting
operation is an automorphism of the Laurent series field
$\mathbb{C}_{\infty}((t))$ that stabilizes several subrings, e.g.,
$\overline{k}[[t]]$, $\overline{k}[t]$, and $\mathbb{T}$. More
generally, for any matrix $B$ with entries in
$\mathbb{C}_{\infty}((t))$ we define $B^{(-n)}$ by the rule
${B^{(-n)}}_{ij}={B_{ij}}^{(-n)}$.

\subsubsection{Entire power series.}

A power series $f=\sum_{i=0}^{\infty}a_{i}t^{i}\in
\mathbb{C}_{\infty}[[t]]$ that satisfies
\[
  \lim_{i\rightarrow \infty}\sqrt[i]{|a_{i}|_{\infty}}=0
\]
and
\[
  [k_{\infty}(a_{0},a_{1},a_{2},\dots):k_{\infty}]< \infty
\]
is called an entire power series. As a function of $t$, such a power
series $f$ converges on all $\mathbb{C}_{\infty}$ and, when
restricted to $\overline{k_{\infty}}$, $f$ takes values in
$\overline{k_{\infty}}$. The ring of the entire power series is
denoted by $\mathbb{E}$.

\subsection{Tannakian categories and difference Galois groups}
We follow \cite{P08} in working with the Tannakian category of
$t$-motives and the Galois theory of Frobenius difference equations,
which is analogous to classical differential Galois theory. In this
subsection, we fix a positive integer $\ell$ and let
$\bar{{\s}}:={\s}^{\ell}$.

Let $\bar{k}(t)[\bar{{\s}},\bar{{\s}}^{-1}]$ be the non-commutative
ring of Laurent polynomials in $\bar{{\s}}$ with coefficients in
$\bar{k}(t)$, subject to the relation
\[
  \bar{{\s}}f:=f^{(-\ell)} \bar{{\s}} \hbox{ for all }f\in \bar{k}(t).
\]A pre-$t$-motive $M$ is a left
$\bar{k}(t)[\bar{\s},\bar{\s}^{-1}]$-module which is
finite-dimensional over $\bar{k}(t)$. Let
$\mathbf{m}\in\hbox{Mat}_{r\times 1}(M)$ be a $\bar{k}(t)$-basis of
$M$. Multiplication by $\bar{\s}$ on $M$ is represented by
$\bar{\s}(\mathbf{m})=\Phi \mathbf{m}$ for some matrix $\Phi \in
\hbox{GL}_{r}(\bar{k}(t))$. Furthermore, $M$ is called rigid
analytically trivial if there exists $\Psi\in
\hbox{GL}_{r}(\mathbb{L})$ such that $\bar{\s}(\Psi)=\Phi \Psi$.
Such a matrix $\Psi$ is called a rigid analytic trivialization of
the matrix $\Phi$.

The category of pre-$t$-motives forms an abelian
${\FF}_{q^{\ell}}(t)$-linear tensor category. Moreover, the category
$\mathcal{R}$ of rigid analytically trivial pre-$t$-motives forms a
neutral Tannakian category over ${\FF}_{q^{\ell}}(t)$  (for the
definition of neutral Tannakian categories, see \cite[Chap.
II]{DMOS82}). For each $M\in \mathcal{R}$, Tannakian duality asserts
that the smallest Tannakian subcategory of $\mathcal{R}$ containing
$M$ is equivalent to the finite dimensional representations of some
affine algebraic group scheme $\G_{M}$ over $\FF_{q^{\ell}}(t)$.

The group $\G_{M}$ is called the Galois group of $M$ and it can be
described explicitly as follows.  Let $\Phi\in
{\rm{GL}}_{r}(\bar{k}(t))$ be the matrix providing multiplication by
$\bar{\s}$ on a rigid analytically trivial pre-$t$-motive $M$ with
rigid analytic trivialization $\Psi\in \hbox{GL}_{r}(\mathbb{L})$.
Let $X:=(X_{ij})$ be an $r\times r$ matrix whose entries are
independent variables $X_{ij}$, and define a
$\overline{k}(t)$-algebra homomorphism $\nu :
\overline{k}(t)[X,1/\det X] \to \mathbb{L}$ so that $\nu(X_{ij}) =
\Psi_{ij}$.  We let
\begin{align*}
\Sigma_\Psi &:= \mathrm{im}\ \nu = \overline{k}(t)[\Psi,1/\det \Psi]
\subseteq \mathbb{L}, \\
Z_\Psi &:= \mathrm{Spec}\ \Sigma_{\Psi}.
\end{align*}
Then $Z_\Psi$ is a closed $\overline{k}(t)$-sub-scheme of
$\mathrm{GL}_{r}/\overline{k}(t)$.  Let $$\Psi_1,\Psi_2 \in
\mathrm{GL}_r(\mathbb{L} \otimes_{\overline{k}(t)} \mathbb{L})$$ be
the matrices satisfying $(\Psi_1)_{ij} = \Psi_{ij} \otimes 1$ and
$(\Psi_2)_{ij} = 1 \otimes \Psi_{ij}$, and let $\widetilde{\Psi} :=
\Psi_1^{-1}\Psi_2$.  Now define an $\FF_{q^\ell}(t)$-algebra
homomorphism $$\mu : \FF_{q^\ell}(t)[X,1/\det X] \to \mathbb{L}
\otimes_{\overline{k}(t)} \mathbb{L}$$ so that $\mu(X_{ij}) =
\widetilde{\Psi}_{ij}$.  Let
\begin{equation} \label{Gamma Psi}
\begin{aligned}
\Delta &:= \mathrm{im}\ \mu, \\
\Gamma_\Psi &:= \mathrm{Spec}\ \Delta.
\end{aligned}
\end{equation}
By Theorems~4.2.11, 4.3.1, and 4.5.10 of \cite{P08}, we obtain the
following theorem.

\begin{theorem} \label{Galois theory} {\rm (Papanikolas
    \cite{P08})} The scheme $\Gamma_\Psi$ is a closed
  $\FF_{q^\ell}(t)$-subgroup scheme of $\mathrm{GL}_{r} /
  \FF_{q^\ell}(t)$, which is isomorphic to the Galois group
  $\Gamma_M$ over $\FF_{q^\ell}(t)$.  Moreover $\Gamma_\Psi$ has the
  following properties:
\begin{enumerate}
\item[(a)] $\Gamma_\Psi$ is smooth over $\overline{\FF_{q}(t)}$ and
  geometrically connected.
\item[(b)] $\dim \Gamma_\Psi = \mathrm{tr.deg}_{\overline{k}(t)}\
  \overline{k}(t)(\Psi)$.
\item[(c)] $Z_\Psi$ is a $\Gamma_\Psi$-torsor over $\overline{k}(t)$.
\end{enumerate}
\end{theorem}

We call $\Gamma_{\Psi}$ the Galois group of the functional equation
$\bar{\s} \Psi=\Phi \Psi$. Here we note that $\Gamma_{\Psi}$ can be
regarded as a linear algebraic group over ${\FF}_{q^{\ell}}(t)$
because of Theorem \ref{Galois theory}(a).

Finally, we review the definition of $t$-motives and the main
theorem of \cite{P08}. Let $\bar{k}[t,\bar{\s}]$ be the
non-commutative subring of $\bar{k}(t)[\bar{\s},\bar{\s}^{-1}]$
generated by $t$ and $\bar{\s}$ over $\bar{k}$. An Anderson
$t$-motive (cf.\ \cite{A86,ABP04}) is a left
$\bar{k}[t,\bar{\s}]$-module $\mathcal{M}$ which is free and
finitely generated both as left $\bar{k}[t]$-module and left
$\bar{k}[\bar{\s}]$-module and which satisfies
\[
  (t-{\th})^{N}\mathcal{M}\subseteq \bar{\s}\mathcal{M},
\]
for $N\in \NN$ sufficiently large.  Let ${\bf{m}}$ be a
$\bar{k}[t]$-basis of $\mathcal{M}$. Multiplication by $\bar{\s}$ on
$\mathcal{M}$ is represented by $\bar{\s}{\bf{m}}=\Phi {\bf{m}} $
for some matrix $\Phi\in \hbox{Mat}_{r}(\bar{k}[t])\cap
\hbox{GL}_{r}(\bar{k}(t))$. By tensoring $\mathcal{M}$ with
$\bar{k}(t)$ over $\bar{k}[t]$, $\mathcal{M}$ corresponds to a
pre-$t$-motive $\bar{k}(t)\otimes_{\bar{k}[t]}\mathcal{M}$ given by
$$\bar{\s}(f\otimes m):=f^{(-\ell)}\otimes \bar{\s}m,\hbox{
}f\in\bar{k}(t), m\in \mathcal{M}.$$  Furthermore, $\mathcal{M}$ is called rigid analytically trivial if there exists $\Psi\in \hbox{Mat}_{r}(\mathbb{T})\cap
\hbox{GL}_{r}(\mathbb{L})$ so that $
  \bar{\s}\Psi:=\Psi^{(-\ell)}=\Phi \Psi$.
Rigid analytically trivial pre-$t$-motives that can be constructed
from Anderson $t$-motives using direct sums, subquotients, tensor
products, duals and internal Hom's, are called $t$-motives. These
$t$-motives form a neutral Tannakian category over
$\FF_{q^\ell}(t)$.  For a $t$-motive $M$ with rigid analytic
trivialization $\Psi$, $\Psi(\th)^{-1}$ is called a period matrix of
$M$. The fundamental Theorem of \cite{P08} can be stated as follows:

\begin{theorem}{\rm{(Papanikolas \cite{P08})}}  \label{tr.deg and dim}
  Let $M$ be a $t$-motive with
  Galois group $\Gamma_{M}$. Suppose that $\Phi\in
  {\rm{GL}}_{r}(\bar{k}(t))\cap {\rm{Mat}}_{r}(\bar{k}[t])$ represents
  multiplication by $\bar{\s}$ on $M$ and that
  $\det\Phi=c(t-{\th})^{s}$, $c\in \bar{k}^{\times}$.  Let $\Psi$ be a
  rigid analytic trivialization of $\Phi$ in
  ${\rm{GL}}_{r}(\mathbb{L})\cap {\rm{Mat}}_{r}(\mathbb{E})$ and let
  $\bar{k}(\Psi(\th))$ be the field generated by the entries of
  $\Psi({\th})$ over $\bar{k}$. Then
\[
  {\rm{dim}}\hbox{ }\Gamma_{M}={\rm{tr.deg}}_{\bar{k}}
  \hbox{ }\bar{k}(\Psi(\th)).
\]
\end{theorem}

\section{Arithmetic gamma values and $t$-motives}
\subsection{Basic properties of arithmetic gamma values}
We are interested in the special values $r! \in k_\infty$ for $r\in
\QQ \cap \ZZ_p$ (see Section 1). We see from the definition that
$r!\in k$ for a non-negative integer $r$.  For $r$ a negative
integer, $r!$ is a $k^{\times}$-multiple of
$\tilde{\pi}:=\tilde{\pi}_{1}$ (defined in Section 3.2)
\cite[p.~34]{T91}, and it is thus transcendental over $k$. Moreover,
for $r\in \QQ\cap (\ZZ_p \setminus \ZZ)$, $r!$ depends up to
multiplication by a factor in $\bar{k}$ only on $r$ modulo $\ZZ$
\cite{T91}.  Hence, without loss of generality we focus on those
$r!$ with $-1 < r < 0$.

Given such an $r$, write $ r=\frac{a}{b}$, where $a$ and $b$ are
integers and $b$ is not divisible by $p$. By Fermat's little theorem
we see that $b$ divides $q^{\ell}-1$ for some $\ell \in \NN$. Hence
$r$ can be written in the form
$$r=\frac{c}{1- q^{\ell}} \hbox{ for some }0<c<q^{\ell}-1 . $$
Write $c=\sum _{i=0}^{\ell-1} c_{i} q^{i} $, with $0 \leq c_{i}< q
$. By the definition of $r!$ we see that
\begin{equation}\label{formula 1}
r!=\prod_{i=0}^{\ell-1} (\frac{q^{i}}{1-q^{\ell}} )!^{c_{i}}.
\end{equation}
Hence, to determine all the algebraic relations among
$$\left\{ (\frac{1}{1-q^{\ell}})!, (\frac{2}{1-q^{\ell}})!,\ldots,(\frac{q^{\ell}-2}{1-q^{\ell}})!     \right\} ,$$
we need only concentrate on these $\ell$ values
$$\left\{ (\frac{1}{1-q^{\ell}})!, (\frac{q}{1-q^{\ell}})!,\ldots,(\frac{q^{\ell-1}}{1-q^{\ell}})!     \right\} .$$

\subsection{The Carlitz $\FF_{q^{\ell}}[t]$-module and its Galois
group}\label{subsection for C ell} For a fixed positive integer
$\ell$, we recall the Carlitz $\FF_{q^{\ell}}[t]$-module, denoted by
$C_{\ell}$, which is given by the $\FF_{q^{\ell}}$-linear ring
homomorphism
$$
\begin{array}{rrcl}
  C_{\ell}: & \FF_{q^{\ell}}[t] &\rightarrow  & {\rm{End}}_{\FF_{q^{\ell}}}(\GG_{a}) \\
   & t & \mapsto & (x\mapsto \th x+x^{q^{\ell}}) .\\
\end{array}
$$
One has the Carlitz exponential,
$${\rm{exp}}_{C_{\ell}}(z)=\sum_{h=0}^{\infty} \frac{z^{q^h}}{D_h} = z \prod_{0\neq a \in \FF_{q^{\ell}}[\th] } (1-\frac{z}{a \tilde{\pi}_{\ell} }   ),   $$
where
\begin{equation}\label{def of tilde pi}
\tilde{\pi}_{\ell}:=\th ( -\th )^{\frac{1}{q^{\ell} -1 }}
\prod_{i=1}^{\infty}(  1-  \frac{\th}{\th^{q^{\ell i}} } )^{-1}
\end{equation}
is a fundamental period of $C_{\ell}$ over $\FF_{q^{\ell}}[t]$.
Throughout this paper we fix a choice of $(-\th)^{\frac{1}{q^{\ell}
-1}}$ so that $\tilde{\pi}_{\ell}$ is a well-defined element in
$\overline{k_{\infty}}$. We also choose these roots in a compatible
way so that when $\ell | \ell'  $ the number
$(-\th)^{\frac{1}{q^{\ell}-1}}$ is a power of
$(-\th)^{\frac{1}{q^{\ell'}-1}}$.

We can regard $C_{\ell}$ also as a Drinfeld $\FF_{q}[t]$-module, and
then it is of rank $\ell$ with complex multiplication by
$\FF_{q^{\ell}}[t]$ (see \cite{G96}, and \cite{T04}). There is a
canonical $t$-motive associated to this Drinfeld $\FF_{q}[t]$-module
$C_{\ell}$, which we denote by $M_{\ell}$.  Its construction is
given below (cf. \cite[Sec. 2.4]{CP08}).

Define $\Phi_{\ell}:=(t-\th)$ if $\ell=1$, and otherwise let
$$
\Phi_{\ell}:=\left[%
\begin{array}{ccccc}
  0 & 1 & 0 & \cdots & 0\\
  0 & 0 & 1 & \cdots & 0 \\
  \vdots & \vdots & \ddots & \ddots & \vdots \\
  0 & 0 & \cdots & 0 & 1 \\
  (t-\th) & 0 & 0 & \cdots & 0 \\
\end{array}%
\right]\in {\rm{GL}}_{\ell}(\bar{k}(t))\cap
{\rm{Mat}}_{\ell}(\bar{k}[t]).
$$ Let $\xi_{\ell}$ be a primitive element of $\FF_{q^{\ell}}$ and
define $\Psi_{\ell}:=\Omega_{\ell}$ if $\ell=1$, and otherwise let
$$\Psi_{\ell}:=\left[%
\begin{array}{cccc}
  \Omega_{\ell} & \xi_{\ell}\Omega_{\ell} & \cdots & \xi_{\ell}^{\ell -1}\Omega_{\ell}   \\
  \Omega_{\ell}^{(-1)} & (\xi_{\ell}\Omega_{\ell} )^{(-1)} & \cdots & ( \xi_{\ell}^{\ell -1}\Omega_{\ell} )^{(-1)} \\
  \vdots & \vdots & \ddots & \vdots \\
  \Omega_{\ell}^{(-(\ell-1))} & ( \xi_{\ell}\Omega_{\ell})^{(-1)} & \cdots & (\xi_{\ell}^{\ell -1}\Omega_{\ell}   )^{(-(\ell-1))} \\
\end{array}%
\right] \in  {\rm{Mat}}_{\ell}(\TT),
$$where \[
\Omega_{\ell}(t):=(-{\th})^{\frac{-q^{\ell}}{q^{\ell}-1}}\prod_{i=1}^{\infty}
\biggl(1-\frac{t}{{\th}^{q^{\ell i}}} \biggr)\in
\overline{k_{\infty}}[[t]]\subseteq \mathbb{C}_{\infty}((t)).
\]
Observe that $\Omega_{\ell}$ is an entire function and that $\Omega_{\ell}(\th)=\frac{-1}{\tilde{\pi}_{\ell}}$.
Moreover, one has the following functional equations,
\begin{equation}\label{fun equ for Omega ell}
\Omega_{\ell}^{(-\ell)}=(t-\th) \Omega_{\ell},
\end{equation}
\begin{equation}\label{func equ for Psi ell}
\Psi_{\ell}^{(-1)}=\Phi_{\ell}\Psi_{\ell}.
\end{equation}
Since $\{1,\xi_{\ell},\ldots,\xi_{\ell}^{\ell-1}\}$ is a basis of
$\FF_{q^{\ell}}$ over $\FF_{q}$, we have that $\Psi_{\ell}\in
{\rm{GL}}_{\ell}(\LL)$.

Let $M_{\ell}$ be an $\ell$-dimensional vector space over
$\bar{k}(t)$ with a basis $\mathbf{m}\in \hbox{Mat}_{\ell \times
1}(M_{\ell})$.  We give $M_{\ell}$ the structure of a left
$\bar{k}(t)[\s,\s^{-1}]$-module by defining
$\s\mathbf{m}:=\Phi_{\ell} \mathbf{m}$, thus making $M_\ell$ a
pre-$t$-motive.  One directly checks that this pre-$t$-motive $M_{\ell}$ is in
fact a $t$-motive with rigid analytic trivialization provided by
$\Psi_{\ell}$.  Working out its
Galois group, we have the following Lemma.

\begin{lemma}\label{Galois froup of M ell}
Let $M_{\ell}$ be the $t$-motive defined above.  Then its Galois
group $\G_{M_{\ell}}\subseteq {\rm{GL}}_{\ell}/ \FF_{q}(t) $ is an
$\ell$-dimensional torus over $\FF_{q}(t)$.
\end{lemma}
\begin{proof}

Let $\bar{\s}:=\s^{\ell}$. Since $M_{\ell}$ is a left
$\bar{k}(t)[\s,\s^{-1}]$-module, it also can be regarded as a left
$\bar{k}(t)[\bar{\s},\bar{\s}^{-1}]$-module given by $
\bar{\s}\mathbf{m}= \widetilde{\Phi_{\ell}}\mathbf{m}$, where
$$\widetilde{\Phi_{\ell}}:=\Phi_{\ell}^{(-(\ell-1))}\cdots \Phi_{\ell}^{(-1)}\Phi_{\ell}.$$
To distinguish the two different roles of $M_{\ell}$, we let
$\mathbf{M}_{\ell}$ have the same underlying space as
$M_{\ell}$, but we regard $\mathbf{M}_{\ell}$ to be a left $\bar{k}(t)[
\bar{\s},\bar{\s}^{-1} ]$-module. Note that $\mathbf{M}_{\ell}$ is a
rigid analytically trivial pre-$t$-motive because
\begin{equation}\label{fun equa of Psi ell for bar s}
\bar{\s}\Psi_{\ell}=\widetilde{\Phi_{\ell}}\Psi_{\ell}.
\end{equation}
Since by Theorem \ref{Galois theory} the Galois group
$\G_{\mathbf{M}_{\ell}}/\FF_{q^\ell}(t)$ of $\mathbf{M}_{\ell}$ is isomorphic to the
Galois group of the functional equation $\eqref{fun equa of Psi ell
for bar s}$, by \eqref{Gamma Psi} we see that
\begin{equation}\label{scalar ext of G M ell}
 \G_{\mathbf{M}_{\ell} }\cong \G_{M_{\ell}}\times_{\FF_{q}(t)}
\FF_{q^{\ell}}(t)  \hbox{ over }\FF_{q^{\ell}}(t).
\end{equation}

Now let $\widetilde{\Psi_{\ell}}$ be the diagonal matrix with entries
$$ \Omega_{\ell},\Omega_{\ell}^{(-1)},\ldots,\Omega_{\ell}^{(-(\ell-1))}.
$$
Since $\widetilde{\Phi_{\ell}}$ is equal to the diagonal matrix with
entries
$$(t-\th),(t-\th)^{(-1)},\ldots,(t-\th)^{(-(\ell-1))},$$
using \eqref{fun equ for Omega ell} we have $\bar{\s}
\widetilde{\Psi_{\ell}}=\widetilde{\Phi_{\ell}}\widetilde{\Psi_{\ell}}$.
That is, $\widetilde{\Psi_{\ell}}$ is also a rigid analytic
trivialization of $\widetilde{\Phi_{\ell}}$ with respect to the
operator $\bar{\s}$ and hence
$$
\G_{\mathbf{M}_{\ell}} \cong \G_{\widetilde{\Psi_{\ell}}}\hbox{ over
}\FF_{q^{\ell}} (t).
$$
On the other hand, since $\widetilde{\Psi_{\ell}}$ is a diagonal
matrix, by \eqref{Gamma Psi} $\G_{\widetilde{\Psi_{\ell}}}$ is a
split torus inside $\hbox{GL}_{\ell}/ \FF_{q^{\ell}}(t)$. Therefore,
from \eqref{scalar ext of G M ell} we see that $\G_{M_{\ell}}$ is a
torus over $\FF_{q}(t)$.

To prove ${\rm{dim}}\hbox{ }\G_{M_{\ell}}=\ell$, it suffices by
Theorem \ref{Galois theory} to show that the transcendence degree of
the field
$\bar{k}(t)(\Psi_{\ell})=\bar{k}(t)(\Omega_{\ell},\Omega_{\ell}^{(-1)},\ldots,\Omega_{\ell}^{(-(\ell-1))})$
over $\bar{k}(t)$ is $\ell$. Now, we let $X_{1},\ldots,X_{\ell}$ be
the coordinates of the $\ell$-dimensional split torus $T$ in
${\rm{GL}_{\ell}}/ \FF_{q^{\ell}}(t)$.  Suppose that
$\G_{\widetilde{\Psi_{\ell}}}$ is a proper subtorus of $T$. Then
$\G_{\widetilde{\Psi_{\ell}}}$ is the kernel of some characters of
$T$, i.e., canonical generators of the defining ideal of
$\G_{\widetilde{\Psi_{\ell}}}$ can be chosen of the form
${X_{1}}^{m_{1}}\cdots {X_{\ell}}^{m_{\ell}}-1$ for integers
$m_{1},\ldots,m_{\ell}$, not all zero. By \eqref{Gamma Psi},
replacing each $X_{i}$ by $ (1/ \Omega_{\ell}^{(-i+1)} ) \otimes
\Omega_{\ell}^{(-i+1)}\in \mathbb{L}\otimes_{\bar{k}(t)}\mathbb{L}$,
we obtain
\begin{equation}\label{equation for beta}
 \prod_{i=1}^{\ell}(\Omega_{\ell}^{(-i+1)})^{m_{i}}
=:\beta \in \bar{k}(t).
\end{equation}
Note that for each $1\leq i \leq \ell$, the set of all zeros of
$\Omega_{\ell}^{(-i+1)}$ is $\left\{ \th^{q^{\ell j-i+1}}
\right\}_{j=1}^{\infty}$.  Since $\beta$ has only finitely many
zeros and poles, comparing the order of vanishing at $t=\th^{q^{\ell
j -i+1}}$ on both sides of \eqref{equation for beta}, for each
$1\leq i \leq \ell$ and for sufficiently large $j$, forces
$m_{i}=0$, and hence we obtain a contradiction.
\end{proof}

Combining Lemma \ref{Galois froup of M ell} and
Theorem \ref{tr.deg and dim}, we also have the following corollary.

\begin{corollary}\label{M ell}
\begin{equation}\label{equality 1}
 {\rm{tr.deg}}_{\bar{k}}\hbox{ } \bar{k}\left( \Omega_{\ell}(\theta),\Omega_{\ell}^{(-1)}(\theta),\ldots,\Omega_{\ell}^{(-(\ell-1))}(\theta) \right) = \ell.
\end{equation}
\end{corollary}

\subsection{Determining algebraic relations for  arithmetic Gamma values}\label{subsec  gamma}  For
nonzero elements $x$, $y\in \CC_{\infty}$, we write $x\sim y$ when
$x/ y \in \bar{k}$.

\begin{theorem}{\rm{(Thakur \cite{T91})}} For each
positive integer $\ell$, we have
\begin{equation}\label{formua 2}
\frac{ (  \frac{1}{1-q^{\ell}}  ) !   }{
(\frac{q^{\ell-1}}{1-q^{\ell}})!^{q} }\thicksim
\Omega_{\ell}(\theta).
\end{equation}
\end{theorem}

\begin{theorem} Fix an integer $\ell \geq 2$.  For each $j$, $ 1\leq j \leq \ell-1$,
we have
\begin{equation}\label{formula 3}
\frac{( \frac{q^{j}}{1-q^{\ell} }    )!}{ (\frac{q^{j-1}}{1-q^{\ell}
})! ^{q} }\thicksim \Omega_{\ell}^{(-(\ell- j))}(\theta).
\end{equation}
\end{theorem}

The first theorem is an analogue (see \cite[Sec. 4.12]{T04}) of the
Chowla-Selberg formula and the second theorem is its quasi-periods
counterpart.  Proofs for both  follow in exactly the same fashion by
straightforward manipulation.  Use
$D_{i}/D_{i-1}^q=(\theta^{q^i}-\theta)$ and take unit parts to
verify that the left side in each formula is the  one-unit part of
the corresponding right side.  Combining  formulas \eqref{formua 2},
\eqref{formula 3} with Corollary \ref{M ell}, we are able to
determine all algebraic relations among those arithmetic gamma
values:

\begin{corollary}\label{standard}
Fix a positive integer $\ell$. Let $L$ be the field over $\bar{k}$
generated by $$ S_{\ell}:= \left\{ (\frac{1}{1-q^{\ell}})!,
(\frac{2}{1-q^{\ell}})!,\ldots,(\frac{q^{\ell}-2}{1-q^{\ell}})!
\right\}.
$$
Then we have
$${\rm{tr.deg}}_{\bar{k}}\hbox{ }L=\ell.  $$
\end{corollary}
\begin{proof}
Suppose $q>2$ or $\ell>1$.  By \eqref{formula 1} we see that
$$\hbox{tr.deg}_{\bar{k}}\hbox{ }L= \hbox{tr.deg}_{\bar{k}} \hbox{ }
\bar{k}\left( (\dfrac{1}{1-q^{\ell}
})!,(\dfrac{q}{1-q^{\ell}})!,\ldots,(\dfrac{q^{\ell-1}}{1-q^{\ell}
})! \right)\leq \ell .$$
From
\eqref{formua 2} and \eqref{formula 3} we observe that
$$ \bar{k}\left( \Omega_{\ell}(\theta),\Omega_{\ell}^{(-1)}(\theta),\ldots,\Omega_{\ell}^{(-(\ell-1))}(\theta) \right)\subseteq L  $$
so that  Corollary \ref{M ell} gives $\hbox{tr.deg}_{\bar{k}}\hbox{ }L \geq \ell $. Thus,
$\hbox{tr.deg}_{\bar{k}}\hbox{ }L=\ell. $

For $q=2$ and $\ell=1$, we interpret $S_1$ as $\{(-1)!\}$ and thus
the Theorem holds in that case too.
\end{proof}

\begin{remark} \label{brackets}
A uniform framework for arithmetic, geometric and
classical gamma functions is described in \cite[Sec. 7]{T91},
\cite[Sec.  4.12]{T04}. In particular, a `bracket criterion' for the
transcendence of gamma monomials at proper fractions is described.
Our result implies that a set of arithmetic gamma monomials is
$\overline{k}$-linearly dependent exactly when the ratio of some
pair of them satisfies the bracket criterion. (The exact parallel
statement is proved for geometric gamma monomials in \cite{ABP04},
see \cite[Thm. 10.5.3]{T04}). In fact, by the proof of \cite[Thm.
4.6.4]{T04} a given arithmetic gamma monomial satisfies the bracket
criterion precisely when, by integral translations of arguments, it
is expressible as a monomial in $(q^j/(1-q^{\ell}))!$'s (with fixed
$\ell$ and $0\leq j< \ell$ and up to an element of $k$) and the
latter monomial is trivial.  Hence Corollary~\ref{standard} implies
that all algebraic relations over $\overline{k}$ among special
arithmetic gamma values are generated by their bracket relations.
\end{remark}

\section{Algebraic independence of Gamma values and Zeta values}

If $q= 2$, the zeta value $\zeta_C(n)$ for any positive integer $n$
is a rational multiple of $\tilde\pi^n$, with $\tilde\pi
=\tilde\pi_1$ (see Section \ref{subsec zeta}). Thus one can easily
determine all the algebraic relations among these special zeta
values and the arithmetic gamma values put together via Section
\ref{subsec  gamma}. The question which remains is the algebraic
independence of $\zeta_C(n)$ along with arithmetic gamma values for
$q> 2$.  Thus we assume $q>2$ throughout this section.

\subsection{The main theorem}

\subsubsection{Carlitz motives and their tensor powers}
For convenience, we let $C:=M_{1}$ be the $t$-motive associated to
the Carlitz $\FF_{q}[t]$-module and let $\Omega:=\Omega_{1}$ (cf.
Section \ref{subsection for C ell}). For each $n\in \NN$, we
introduce the $n$-th tensor power $C^{\otimes n}$ of the Carlitz
motive $C$. Its underlying space is $\bar{k}(t)$ with $\s$-action $
\s f:= (t-\th)^{n} f^{(-1)} \hbox{ for }f\in C$. Thus $\Omega^{n}$
provides a rigid analytic trivialization for $C^{\otimes n}$.  The Galois group
of  $C^{\otimes n}$ is isomorphic to
$\GG_{m}$ over $\FF_{q}(t)$ because $\Omega$ is transcendental over
$\bar{k}(t)$ (cf. Theorem \ref{Galois theory}).

\subsubsection{Polylogarithms and $L_{n,\a}(t)$}
Recall the Carlitz logarithm ${\rm{log}}_{C}(z)$ of the Carlitz
$\FF_{q}[t]$-module. As a power series, it is the inverse of
${\rm{exp}}_{C_{1}}(z)$ with respect to composition, and it can be
written as
$${\rm{log}}_{C}(z)=z+\sum_{i=1}^{\infty} \frac{z^{q^{i}}}{(\th-\th^{q})(\th-\th^{q^{2}})\cdots (\th-\th^{q^{i}})        }       .$$
For each $n\in \NN$, the $n$-th polylogarithm is defined by
$${\rm{log}}^{[n]}_{C}(z):= z+\sum_{i=1}^{\infty} \frac{z^{q^{i}}}{(\th-\th^{q})^{n}(\th-\th^{q^{2}})^{n}\cdots (\th-\th^{q^{i}})^{n}        }          .$$
It converges on the disc $\{  |z|_{\infty}< |\th|_{\infty}^{
\frac{nq}{q-1}} \}$.

For $n\in \NN$ and $\a\in \bar{k}^{\times}$ with $|\a|_{\infty}<
|\th|_{\infty}^{\frac{nq}{q-1}}$, we consider the power series
$$ L_{\a,n}(t):= \a+\sum_{i=1}^{\infty} \frac{\a^{q^{i}}}{(t-\th^{q})^{n}(t-\th^{q^{2}})^{n}\cdots (t-\th^{q^{i}})^{n}        }     ,  $$
which converges on the disc $\{ |t|_{\infty}< |\th|_{\infty}^{q} \}$
and satisfies $L_{\a,n}(\th)={\rm{log}}^{[n]}_{C}(\a)$.
Moreover, the action of $\s$ on $L_{\a,n}(t)$ gives rise to
$$ L_{\a,n}^{(-1)}(t)=\a^{(-1)}+\frac{L_{\a,n}(t)}{(t-\th)^{n}},$$
and hence we have the functional equation
\begin{equation}\label{f.e.L}
 (\Omega^{n} L_{\a,n})^{(-1)}=\a ^{(-1)}(t-\th)^{n}\Omega^{n}
+\Omega^{n}L_{\a,n}.
\end{equation}

\subsubsection{Review of the Chang-Yu theorem}
Fixing a positive integer $s$, we define
$$
U(s):=\{ n\in \NN;\hbox{ }1\leq n\leq s,\hbox{ }p\nmid n,\hbox{ }
(q-1)\nmid n     \}.
$$
For each $n\in U(s)$ we fix a finite set of $1+m_{n}$ elements
$$\{  \a_{n0},\ldots,\a_{n m_{n}} \} \subseteq \bar{k}^{\times}  $$
so that
\begin{equation}\label{condition 1 for alpha}
 |\a_{nj}|_{\infty} < |\th|_{\infty}^{\frac{n q}{q-1}} \hbox{ for
} j=0,\ldots,m_{n},
\end{equation}
\begin{equation}\label{condition 2 for alpha}
 \tilde{\pi}^{n}, \mathcal{L}_{n0}(\th) ,\cdots,\mathcal{L}_{n
m_{n}}(\th)\hbox{ are linearly independent over }k,
\end{equation}
where $$\mathcal{L}_{nj}:=L_{\a_{j},n } \hbox{ for }
j=0,\ldots,m_{n}.$$ For $n\in U(s)$, define $\Phi_n \in
\hbox{GL}_{m_{n}+2}(\overline{k}(t))\cap
\hbox{Mat}_{m_{n}+2}(\overline{k}[t])$ and $\Psi_n \in
\hbox{GL}_{m_{n}+2}({\TT})$ by
\[
\Phi_{n}=\Phi({\a}_{n0},\ldots,{\a}_{n m_{n}}) :=\left[%
\begin{array}{cccc}
  (t-{\th})^{n} & 0 & \cdots & 0 \\
  {\a}_{n 0}^{(-1)}(t-{\th})^{n} & 1 & \cdots & 0 \\
  \vdots & \vdots & \ddots & \vdots \\
  {\a}_{n m_{n}}^{(-1)}(t-{\th})^{n} & 0 & \cdots & 1 \\
\end{array}%
\right],
\]
\[
\Psi_{n}=\Psi({\a}_{n0},\ldots,{\a}_{nm_{n}}):= \left[%
\begin{array}{cccc}
  \Omega^{n} & 0 & \cdots & 0 \\
  \Omega^{n}\mathcal{L}_{n0} & 1 & \cdots & 0 \\
  \vdots & \vdots & \ddots & \vdots \\
  \Omega^{n}\mathcal{L}_{n m_{n}} & 0 & \cdots & 1 \\
\end{array}%
\right].
\]
Then by \eqref{f.e.L} we have $\Psi_{n}^{(-1)}=\Phi_{n} \Psi_{n}$.
Note that all the entries of $\Psi_{n}$ are inside $\mathbb{E}$ and
that $\Phi_{n}$ defines a $t$-motive which is an extension of
the $(m_{n}+1)$-dimensional trivial $t$-motive over $\bar{k}(t)$ by
$C^{\otimes n}$ (cf. \cite{CY07} Lemma A.1).

Fixing a positive integer $s$, we define the block diagonal
matrices,
$$\Phi_{(s)}:=\oplus_{n\in U(s)} \Phi_{n},$$
$$\Psi_{(s)}:=\oplus_{n\in U(s)} \Psi_{n}.$$
Each $\Phi_{(s)}$ defines a $t$-motive $M_{(s)}$ which is the direct
sum of the $t$-motives defined by $\Phi_{n}$, $n\in U(s)$. Using
(\ref{Gamma Psi}), we see that any element in $\G_{\Psi_{(s)}}$ is
of the form
$$\oplus_{n\in U(s)} \left[%
\begin{array}{cccc}
  x^{n} & 0 & \cdots & 0 \\
  * & 1 & \cdots & 0 \\
  \vdots & \vdots &  \ddots& \vdots \\
  * & 1 & \cdots & 1 \\
\end{array}%
\right],$$ where the block matrix at the position corresponding to
$n\in U(s)$ has size $m_{n}+2$.

Since the Carlitz motive $C$ is a sub-$t$-motive of $ M_{(s)}$, by
Tannakian category theory we have a natural surjective map
\begin{equation}\label{surjection of pi}
\pi : \G_{\Psi_{(s)}} \twoheadrightarrow \GG_{m},
\end{equation}
which coincides with the projection on the upper left corner of any
element of $\G_{\Psi_{(s)}}$ (cf.\ \cite[Sec.~4.3]{CY07}). Let
$V_{(s)}$ be the kernel of $\pi$ so that one has an exact sequence
of linear algebraic groups, $1\rightarrow V_{(s)}\rightarrow
\Gamma_{\Psi_{(s)}}\twoheadrightarrow {\GG}_{m}\rightarrow 1$. From
the projection map $\pi$, we see that $V_{(s)}$ is contained in the
$(\sum_{n\in U(s)} (m_{n}+1 ))$-dimensional vector group $G_{(s)}$
in $\G_{\Psi_{(s)}}$ which consists of all block diagonal matrices
of the form
$$ \oplus _{n\in U(s)} \left[%
\begin{array}{cccc}
  1 & 0 & \cdots & 0 \\
  \ast & 1 & \cdots & 0 \\
  \vdots & \vdots & \ddots & \vdots \\
  \ast & 0& \cdots & 1 \\
\end{array}%
\right],
$$
where the block matrix at the position corresponding to $n\in U(s)$
has size $m_{n}+2$. Here we shall note that $G_{(s)}$ can be
identified with the direct product $\prod_{n\in U(s)}
{\GG}_{a}^{m_{n}+1}$, with the block matrix corresponding to $n\in
U(s)$ identified with points in $\GG_{a}^{m_{n}+1}$ having
coordinates $(x_{n0},\ldots,x_{n m_{n}})$.

\begin{theorem}{\rm{(Chang-Yu \cite{CY07})}}\label{Chang-Yu thm}
Fix any $s\in \NN$, let $\Phi_{(s)}$, $\Psi_{(s)}$, $V_{(s)}$, and
$G_{(s)}$ be defined as above.  Then we have $V_{(s)}=G_{(s)}$, and
hence
$${\rm{dim}}\hbox{ }\G_{\Psi_{(s)}}=1+\sum_{n\in U(s)} (m_{n}+1).
$$ In particular, the
union
$$ \{ \tilde{\pi} \} \cup_{n\in U(s)} \cup_{i=0}^{m_{n}}\{ \mathcal{L}_{ni}(\th)  \}  $$
is an algebraically independent set over $\bar{k}$.
\end{theorem}

\subsubsection{The main theorem}
Given positive integers $\ell$ and $s$, we consider the $t$-motive
$M:=M_{(s)}\oplus M_{\ell} $, where $M_{\ell}$ is the $t$-motive
defined by $\Phi_{\ell}$ with rigid analytic trivialization
$\Psi_{\ell}$ (see Section \ref{subsection for C ell}). More
precisely, $M$ is defined by $\Phi:=\Phi_{(s)}\oplus \Phi_{\ell}$
with a rigid analytic trivialization $\Psi:=\Psi_{(s)}\oplus
\Psi_{\ell}$. The main theorem of this subsection can be stated as
follows.

\begin{theorem}\label{main thm}
Given any two positive integers $s$ and $\ell$, let $(M,\Phi,\Psi)$
be defined as above. Then the dimension of the Galois group
$\G_{\Psi}$ of $M$ is
$$\ell+ \sum_{n\in U(s)}\{m_{n}+1 \}.$$
In particular, the following set
$$ \cup_{n\in U(s)}\cup_{j=0}^{m_{n}}  \{ \mathcal{L}_{nj}(\th) \} \cup
\{ (\frac{1}{1-q^{\ell}})!,
(\frac{q}{1-q^{\ell}})!,\ldots,(\frac{q^{\ell-1}}{1-q^{\ell}})! \}$$
is a transcendence basis of $\bar{k}(\Psi(\th))$ over $\bar{k}$.
\end{theorem}

\begin{proof}
First we note that by (\ref{Gamma Psi}) any element of
$\G_{\Psi}(\overline{\FF_{q}(t)})$ is of the form
$$ \oplus_{n\in U(s)}  \left[%
\begin{array}{cccc}
  x^{n} & 0 & \cdots & 0 \\
  \ast & 1 & \cdots & 0 \\
  \vdots & \vdots & \ddots & \vdots \\
  \ast & 0& \cdots & 1 \\
\end{array}%
\right]\oplus B , $$ where
$$ \oplus_{n\in U(s)}  \left[%
\begin{array}{cccc}
  x^{n} & 0 & \cdots & 0 \\
  \ast & 1 & \cdots & 0 \\
  \vdots & \vdots & \ddots & \vdots \\
  \ast & 0& \cdots & 1 \\
\end{array}%
\right]\in \G_{\Psi_{(s)}} (\overline{\FF_{q}(t)}),\hbox{ } B\in
\G_{\Psi_{\ell}}(\overline{\FF_{q}(t)}).
$$
Define $\Phi_{D}:=\oplus_{n\in U(s)}[( t-\th )^{n} ]\oplus
\Phi_{\ell}$, $\Psi_{D}:=\oplus_{n\in U(s)} [\Omega^{n}]\oplus
\Psi_{\ell}$. Then we have $\Psi_{D}^{(-1)}=\Phi_{D}\Psi_{D}$, and
we note that such $\Phi_{D}$ defines a $t$-motive $M_{D}$, which is
the direct sum of the $t$-motives $\oplus_{n\in U(s)} C^{\otimes n}
\oplus M_{\ell}$.  Moreover, the $t$-motive $M_{D}$ is a
sub-$t$-motive of $M$, and hence using the same argument as for the
surjection of $\pi$ \eqref{surjection of pi} we have a surjective
map
$$\pi_{D}:\G_{\Psi } \twoheadrightarrow \G_{\Psi_{D}},$$
which coincides with the projection map given by
$$\oplus_{n\in U(s)}  \left[%
\begin{array}{cccc}
  x^{n} & 0 & \cdots & 0 \\
  \ast & 1 & \cdots & 0 \\
  \vdots & \vdots & \ddots & \vdots \\
  \ast & 0& \cdots & 1 \\
\end{array}%
\right]\oplus B \mapsto \oplus_{n\in U(s)}[x^{n}]\oplus B.$$ Put
$V:=\hbox{Ker }\pi_{D}$ and note that $V$ is a vector group.

Since ${\rm{det}}\hbox{ }\Psi_{\ell}(\th)\sim \Omega(\th)$, we have
$\bar{k}(\Psi_{D}(\th))=\bar{k}(\Psi_{\ell}(\th))$, and hence by
Theorem \ref{Galois theory} and Lemma \ref{Galois froup of M ell} we
see that $ \hbox{dim }\Gamma_{\Psi_{D}}=\hbox{dim
}\Gamma_{\Psi_{\ell}}= \ell$.  Therefore, to prove this theorem it
is equivalent to prove that $\hbox{dim }V=\sum_{n\in
U(s)}(m_{n}+1)$.

Consider the following commutative diagram:
\begin{equation}\label{diagram}
\xymatrix{
1\ar[r] & V \ar[r] \ar@{->}[d] & \Gamma_{\Psi} \ar@{->>}[r]^{\pi_{D}} \ar@{->>}[d]^{\pi_{s}} &\G_{\Psi_{D}} \ar[r] \ar@{->>}[d]^{\chi_{D}} & 1\\
1\ar[r]& V_{(s)} \ar[r]& \Gamma_{\Psi_{(s)}} \ar@{->>}[r]^{\pi} &
{\GG}_{m} \ar[r]& 1, }
\end{equation}
where the right hand square is given
by
\[
\xymatrix{ {\displaystyle \oplus_{n\in U(s)}\left[
\begin{matrix}
x^{n} & 0 & \cdots & 0 \\
* & 1 & \cdots & 0 \\
\vdots & \vdots & \ddots & \vdots \\
* & 0 & \cdots & 1
\end{matrix}
\right]\oplus B} \ar@{|->}[d]_{\pi_s}
\ar@{|->}[r]^{\hspace*{30pt}\pi_D}
& {\oplus_{n\in U(s)}[x^n]\oplus B} \ar@{|->}[d]_{\chi_D} \\
{\oplus_{n\in U(s)}\left[
\begin{matrix}
x^{n} & 0 & \cdots & 0 \\
* & 1 & \cdots & 0 \\
\vdots & \vdots & \ddots & \vdots \\
* & 0 & \cdots & 1
\end{matrix}
\right]} \ar@{|->}[r]_{\hspace*{30pt}\pi} & [x]. }
\]
Here we note that the projection maps $\pi_{s}$ and $\pi_{D}$ are
surjective since $M_{(s)}$ is a sub-$t$-motive of $M$ and the
Carlitz motive $C$ is a sub-$t$-motive of $M_{D}$ (cf.
(\ref{surjection of pi})).

By Theorem \ref{Chang-Yu thm} we have $V_{(s)}=G_{(s)}$, which is
identified with the product space $ \prod_{n\in U(s)}
\GG_{a}^{m_{n}+1}$ canonically. For a double index $nj$ with $n\in
U(s)$, $0\leq j \leq m_{n}$, we let the $(nj)$-coordinate space in
$V_{(s)}$ be the one-dimensional vector subgroup consisting of
points whose coordinates all vanish except the coordinate $x_{nj}$.
Hence, we need only show that $\pi_{s}|_{V}$ is surjective onto each
$(nj)$-coordinate space in $V_{(s)}$ for $n\in U(s)$, $0\leq j\leq
m_{n}$.

Given a nonzero element $v$ of the $(nj)$-coordinate space in
$V_{(s)}(\overline{\FF_{q}(t)})$, we can pick $\gamma\in
\G_{\Psi}(\overline{\FF_{q}(t)})$ so that $\pi_{s}(\gamma)=v$ since
$\pi_{s}$ is surjective. Further, pick $a\in
\overline{\FF_{q}(t)}^{\times}\setminus \overline{\FF_{q}}^{\times}$
and let $\delta\in \Gamma_{\Psi}(\overline{\FF_{q}(t)})$ for which
$$\chi_{D}\circ \pi_{D}(\delta)=a . $$
Using the property that $\Gamma_{\Psi_{\ell}}$ is commutative
(cf. Lemma \ref{Galois froup of M ell}), direct calculation shows
that
$$\delta ^{-1} \gamma \delta \gamma^{-1} \in V(\overline{\FF_{q}(t)})  $$
and $ \pi_{s}( \delta ^{-1} \gamma \delta \gamma^{-1})$ belongs to
the $(nj)$-coordinate space in $V_{(s)}(\overline{\FF_{q}(t)} )$.
Moreover, for such choice of $a$, we see that $ \pi_{s}(\delta ^{-1}
\gamma \delta \gamma^{-1} ) $ is nonzero and complete the
proof.
\end{proof}

\subsection{Application to zeta values}\label{subsec zeta}
The key for applying Section 4.1 to our problems on zeta values is the following fact concerning the special zeta value
$\zeta_{C}(n)$ and $n$-th polylogarithms.

\begin{theorem}
{\rm{(Anderson-Thakur \cite{AT90})}} \label{Anderson-Thakur} Given a
positive integer $n$, one can explicitly find a finite sequence
$h_{n,0},\ldots,h_{n,l_{n}}\in k$, $l_{n}<\frac{nq}{q-1}$, such that
\begin{equation}\label{gamma-zeta formula}
\zeta_{C}(n)=\sum_{i=0}^{l_{n}}h_{n,i}L_{{\th}^{i},n}({\th}).
\end{equation}
\end{theorem}

Given any positive integer $n$ not divisible by $q-1$, set
$$ N_{n}:=k\hbox{-span}\{\tilde{\pi}^{n},L_{1,n}({\th}),L_{{\th},n}({\th}),\ldots,
L_{{\th}^{l_{n}},n}({\th}) \}. $$By (\ref{gamma-zeta formula}) we
have $\zeta_{C}(n)\in N_{n}$ and $m_{n}+2:=\hbox{dim}_{k}\hbox{
}N_{n}\geq 2$ since $\zeta_{C}(n)$ and $\tilde{\pi}^{n}$ are
linearly independent over $k$. For each such $n$ we fix once and for
all a finite subset
$$\{ \a_{n0},\ldots,\a_{nm_{n}}\}\subseteq  \{1,\th,\ldots,{\th}^{l_{n}} \}  $$
such that both
$$\{ \tilde{\pi}^{n},\mathcal{L}_{n0}(\th),\ldots,
\mathcal{L}_{nm_{n}}({\th}) \}$$ and $$\{
\tilde{\pi}^{n},\zeta_{C}(n),\mathcal{L}_{n1}({\th}),\ldots,
\mathcal{L}_{nm_{n}}({\th}) \}$$ are bases of $N_{n}$ over $k$,
where $\mathcal{L}_{nj}(t):=L_{{\a}_{nj},n }(t)$ for
$j=0,\ldots,m_{n}$. This can be done because of Theorem
\ref{Anderson-Thakur}.

Given a positive integer $s$, Theorem \ref{Chang-Yu thm} then
implies that all the zeta values $\zeta_C(n)$,  $n\in U(s)$, and
$\tilde\pi$, are algebraically independent over $k$. In other words,
we have
\begin{equation} \label{CYtrdeg}
\hbox{tr.deg}_{\bar{k}}\ \bar{k}(\tilde{\pi},
\zeta_{C}(1),\ldots,\zeta_{C}(s) )=  s-\lfloor s/p \rfloor -\lfloor
s/(q-1)  \rfloor +\lfloor s/p(q-1) \rfloor +1.
\end{equation}
This also implies that all $\bar{k}$-algebraic relations among
Carlitz zeta values are those relating the $\zeta_C(m)$, with
$m\notin U(s)$, and those $\zeta_{C}(n)$ with $n\in U(s)$ for given
$s\in \mathbb{N}$.

These relations come from the Frobenius $p$-th power relations and
the Euler-Carlitz relations, which we recall briefly.  The Frobenius
$p$-th power relations among zeta values ($p$ is the characteristic)
are
$$ \zeta_{C}(p^{m}n)=\zeta_{C}(n)^{p^{m}} \hbox{ for }m,n\in \mathbb{N}.$$
Also the Euler-Carlitz relations among the $\zeta_C(n)$, for $n$
divisible by $q-1$, and $\tilde{\pi}$ are
$$\zeta_{C}(n)=\frac{B_{n}}{\G_{n+1}} \tilde{\pi}^{n}.$$
The Bernoulli-Carlitz `numbers' $B_{n}$ in $k$ are given by the
following expansion from the Carlitz exponential series,
\[
\frac{z}{\hbox{exp}_{C_{1}}(z)}=\sum_{n=0}^{\infty}
\frac{B_{n}}{\G_{n+1}}z^{n},
\]
where $\G_{n+1}$ is the Carlitz factorial of $n$  (cf.\ Section 1).
Thus the formula in \eqref{CYtrdeg} is obtained by
inclusion-exclusion.

Now, applying Theorem \ref{main thm} to this setting we can
determine the transcendence degree of the field generated by all
arithmetic gamma values and  zeta values put together.

\begin{theorem}\label{main result}
Given any two positive integers $s$ and $\ell$, let $E$ be the field
over $\bar{k}$ generated by the set
$$\left\{ \tilde{\pi},\zeta_{C}(1),\ldots,\zeta_{C}(s) \right\}
\bigcup \left\{ (\frac{c}{1-q^{\ell} })!;\hbox{ }1\leq c \leq
q^{\ell}-2
  \right\}.
$$Then the transcendence degree of $E$ over $\bar{k}$ is
$$ s-\lfloor s/p \rfloor -\lfloor s/(q-1)  \rfloor +\lfloor s/p(q-1) \rfloor + \ell.            $$
\end{theorem}

\begin{remark}
In the classical case, a conjecture about the Riemann zeta function
at positive integers asserts that the Euler relations, i.e.,
$\zeta(2n)/ (2 \pi \sqrt{-1})^{2n}\in \mathbb{Q}$ for $n\in
\mathbb{N}$, account for all the $\overline{\mathbb{Q}}$-algebraic
relations among the special zeta values $\left\{
\zeta(2),\zeta(3),\zeta(4),\ldots \right\}$. For special
$\G$-values, i.e., values of the Euler $\G$-function at proper
fractions, there are the `natural' $\overline{\mathbb{Q}}$-algebraic
relations among them coming from the translation, reflection and
Gauss multiplication identities satisfied by the $\G$-function,
which are referred to as the standard relations. The Rohrlich-Lang
conjecture predicts that these relations account for all
$\overline{\mathbb{Q}}$-algebraic relations among special
$\G$-values.  One is lead to the conjectures that  the Euler
relations among special zeta values and the standard algebraic
relations among special $\G$-values  account for all the
$\overline{\mathbb{Q}}$-algebraic relations among the special zeta
values and special $\G$-values put together.
\end{remark}

\begin{remark}
Our Corollary \ref{standard} asserts that the standard relations
among the arithmetic gamma values account for all
$\bar{k}$-algebraic relations among the arithmetic gamma values.
Theorem \ref{main result} asserts that the Euler-Carlitz relations,
the Frobenius $p$-th power relations and the standard relations
among the arithmetic gamma values account for all
$\bar{k}$-algebraic relations among the arithmetic gamma values and
Carlitz zeta values put together.
\end{remark}

\end{document}